\documentclass[12pt, ams fonts]{amsart}

\usepackage{color}
\usepackage{stmaryrd}

\usepackage{amssymb,amsmath}

\usepackage{graphicx}

\usepackage[vcentermath,enableskew]{youngtab}

\input xy
\xyoption{all}

\font\teneufm=eufm10 \font\seveneufm=eufm7 \font\fiveeufm=eufm5
\newfam\eufmfam
\textfont\eufmfam=\teneufm \scriptfont\eufmfam=\seveneufm
\scriptscriptfont\eufmfam=\fiveeufm

\newcommand{\C}{\mathbb{C}}

\newcommand{\Z}{\mathbb{Z}}

\newcommand{\g}{\mathfrak{g}}
\newcommand{\h}{\mathfrak{h}}

\DeclareMathOperator{\Supp}{Supp}
\DeclareMathOperator{\Span}{Span} \DeclareMathOperator{\tr}{tr}

 \DeclareMathOperator{\ad}{ad}
 \DeclareMathOperator{\Aut}{Aut}

\DeclareMathOperator{\Tr}{Tr}

\DeclareMathOperator{\Der}{Der}
\DeclareMathOperator{\End}{End}

\numberwithin{equation}{section}

\newtheorem{definition}{Definition}[section]

\theoremstyle{remark}

\newtheorem{remark}[definition]{Remark} 

\theoremstyle{plain} 

\newtheorem{theorem}[definition]{Theorem}
\newtheorem{lemma}[definition]{Lemma}
\newtheorem{corollary}[definition]{Corollary}
\newtheorem{proposition}[definition]{Proposition}

\def\Z{\mathbb Z}
\def\C{\mathbb C}

\begin{document}

\title{Exponentiation and Fourier transform of tensor modules of $\mathfrak{sl} (n+1)$}

\author{Dimitar Grantcharov and Khoa Nguyen}

\address{Department of Mathematics \\
         University of Texas at Arlington \\ Arlington, TX 76021, USA}

         \email{khoa.nguyen2@uta.edu}
         
\address{Department of Mathematics \\
         University of Texas at Arlington \\ Arlington, TX 76021, USA}

         \email{grandim@uta.edu}

\thanks{This work was partially supported by Simons Collaboration Grant 358245}

\maketitle

\begin{abstract}
With the aid of the exponentiation functor and Fourier transform we introduce a class of modules $T(g,V,S)$ of $\mathfrak{sl} (n+1)$ of mixed tensor type. By varying the polynomial $g$, the $\mathfrak{gl}(n)$-module $V$, and the set $S$, we obtain important classes of  weight modules over the Cartan subalgebra $\mathfrak h$ of $\mathfrak{sl} (n+1)$, and  modules that are free over $\mathfrak h$.  Furthermore, these modules are obtained through explicit presentation of the elements of $\mathfrak{sl} (n+1)$ in terms of differential operators and lead to new tensor coherent families of  $\mathfrak{sl} (n+1)$. An isomorphism theorem and simplicity criterion for $T(g,V,S)$ is provided. 

\medskip\noindent 2020 MSC: 17B10, 17B66 \\

\noindent Keywords and phrases: Lie algebra, tensor module, weight module

\end{abstract}

\section*{Introduction} 
There are two important, but opposite in nature, categories of modules of finite-dimensional reductive Lie algebras $\mathfrak a$. The first one consists of weight modules, namely, those that decompose into direct sums of their weight spaces relative to a fixed Cartan subalgegra $\mathfrak h$. The second one is the category of $\mathfrak h$-free modules.  The classification of the simple objects in these categories is far from reach unless one imposes an additional finiteness condition. In particular, simple weight $\mathfrak a$-modules with finite weight multiplicities   have been classified by O. Mathieu, \cite{M}, following works of  G. Benkart, D. Britten, S. Fernando, V. Futorny, A. Joseph, F. Lemire, and others. On the other hand, the classification of all simple  $\mathfrak h$-free modules of finite-rank is still an open problem and the only known case is when the rank equals one, \cite{Nil1}.

A crucial role in Mathieu's breakthrough paper, \cite{M} plays the new notion of \emph{coherent family}  -  a ``big'' weight module whose support coincides with the whole $\mathfrak h^*$. Coherent families have explicit geometric realizations via sections of vector bundles of algebraic varieties (called tensor coherent families), and also can be constructed purely algebraically through twisted localization of highest modules. The geometric realization is especially convenient in the case of $\mathfrak{sl} (n+1)$, when the modules  are tensor products $T(P,V) = P\otimes V$ of mixed type. More precisely, $P$ is a module over the algebra $\mathcal  D (n)$ of polynomial differential operators of $\mathcal O = {\mathbb C} [t_1,...,t_n]$ and $V$ is a module over the Lie algebra  $\mathfrak{gl}(n)$. As a result, we have an explicit presentation of the root elements of $\mathfrak{sl} (n+1)$ in terms of differential operator presentation. The tensor modules of mixed type $T(P,V)$ were  introduced by Shen, \cite{Sh} and Rudakov, \cite{R}, and play important role in the representation theory of various Lie algebras of derivations and vector fields, see for example, \cite{BF}, \cite{HCL}, \cite{LLZ}, \cite{TZ}, \cite{TZ2}.

Throughout the paper we fix $\mathfrak s = \mathfrak{sl} (n+1)$. One of the tools used in the study of $\mathfrak h$-free $\mathfrak s$-modules is the weighting functor $\mathcal W$. This functor maps an $\mathfrak h$-free module $M$ of finite rank to a coherent family $\mathcal M$ and raises a natural question about further connections between the categories of weight and $\mathfrak h$-free modules. A main purpose of the present paper is to make such connection and in particular to combine both types of modules together. This is done thanks to applying two functors on the tensor modules $T(\mathcal O,V)$. The two functors are exponentiation $\exp_g$ by a polynomial $g$, and a Fourier transform $\psi_S$  relative to a subset $S$ of $\{1,2,...,n\}$. As a result we define  the \emph{exponential tensor modules} $T(g,V,S)$. The case when $g=0$ and improper set $S$ relates to Mathieu's tensor coherent families. When $g=0$, by varying $S$ we obtain all injective partly-irreducible coherent families. In particular, every simple bounded $\mathfrak s$-module appears as a submodule of some coherent family of general type.

The case when $g$ has degree $1$ and $V$ is 1-dimensional leads to the complete list of $\mathfrak h$-free modules of rank $1$. Furthermore, we obtain interesting classes of Whittaker modules and weight modules relative to other Cartan subalgebras. Also, connections with the weighting functor  and Witten deformation of the de Rham complex are discovered. One of the main results  in the present paper are a simplicity criterion and an isomorphism theorem for $T(g,V,S)$. In particular, we provide new families of simple non-weight modules over $\mathfrak{sl} (n+1)$ obtained through explicit presentation of the Lie algebra in terms of differential operators. The main tools used in the proofs of these two results  are the twisted localization and translation functors.

 The content of the paper is as follows. In Section 2 we collect some preliminary results on the algebra $\mathcal D (n)$, the twisted localization functor, coherent families, and $\mathfrak h$-free modules. In particular, we obtain an explicit presentation of $\mathfrak s$ in terms of differential operators that depends on a set $S$ and a $\mathfrak{gl}(n)$-module $V$, and use this presentation to define  the exponential tensor modules $T(g,V,S)$. In Section 3, we use the localization technique to obtain  the simplicity criterion for $T(g,V,S)$ in the case of $1$-dimensional $V$ and $V = \bigwedge^k {\mathbb C}^n$. Using translation functors, in Section 4, we give a necessary and sufficient condition for $T(g,V,S)$ to be simple and also establish the conditions when two exponential tensor modules are isomorphic. Section 5 is devoted to applications of the exponential tensor modules in some particular cases. In the the case $g =0$ we obtain  injective coherent families. In the case $\deg g =1$ we discuss $\mathfrak h$-free modules, weight modules relative to different Cartan subalgebra, and Whittaker modules. \medskip

\section{Notation and Conventions}

Throughout the paper the ground field is $\mathbb C$ and $\mathbb C^* = \mathbb C  \setminus \{0 \}$. All vector spaces, algebras, and tensor products are assumed to be over $\mathbb C$ unless otherwise stated. 
We set $\llbracket k \rrbracket = \{1,2,...,k \}$ for a positive integer $k$. Throughout the paper, ${\mathfrak s} = \mathfrak{sl} (n+1)$ and $\g \simeq \mathfrak{gl} (n)$ is a fixed subalgebra of $\mathfrak s$ defined in Section \ref{sec-tran-fun}.  For a Lie algebra $\mathfrak a$ by $U(\mathfrak a)$ we denote the universal enveloping algebra of $\mathfrak a$ and  by $Z(\mathfrak a)$ the center of $U(\mathfrak a)$. 

Throughout the paper, $\mathcal O = \C [t_1,...,t_n]$ and ${\mathcal O}_0$ will stand for the maximal ideal of $\mathcal O$ generated by $t_1,...,t_n$. By ${\mathcal D} (n)$ we denote the algebra of polynomial differential operators of $\mathcal O$. We set $\partial_i :=\frac{\partial}{\partial t_i}$ and use the notation $t_i$ for the element in $ \End (\mathcal O)$ corresponding to multiplication by $t_i$. In particular, ${\mathcal D} (n)$ is the associative subalgebra of $ \End (\mathcal O)$ generated by $t_i$, $\partial_i$, $i=1,...,n$, subject to 
$$
t_it_j-t_jt_i = \partial_i \partial_j- \partial_j \partial_i = 0; \; \partial_i t_j- t_j \partial_i = \delta_{ij}.
$$

By $W_n$ we denote the Lie algebra $ \Der ( \mathcal O )$ of derivations of $\mathcal O$. Every element $w$ of $W_n$ can be written uniquely as $w= \sum_{i=1}^n f_i\partial_i$, for some $f_i \in \C[x_1,...,x_n]$.

Throughout the paper we use the multi-index notation. In particular, $\boldsymbol{t}^{\boldsymbol{\nu}} = t_1^{\nu_1}...t_n^{\nu_n}$, where 
$\boldsymbol{t} = (t_1,...,t_n)$ and $\boldsymbol{\nu} = (\nu_1,...,\nu_n)$. If $n$ is fixed, we set ${\mathbb C}[\boldsymbol{t}] = {\mathbb C}[t_1,...,t_n] = \mathcal O$, ${\mathbb C}[\boldsymbol{t}^{\pm 1}] = {\mathbb C}[t_1^{\pm 1},...,t_n^{\pm 1}]$, and $\boldsymbol{t}^{\boldsymbol{\nu}}  {\mathbb C}[\boldsymbol{t}^{\pm 1}] = t_1^{\nu_1}...t_n^{\nu_n}{\mathbb C}[t_1^{\pm 1},...,t_n^{\pm 1}]$, where the latter is the span of all (formal) monomials $t_1^{\nu_1+k_1}...t_n^{\nu_n + k_n}$, $k_i \in {\mathbb Z}$. 

	By $S_k$ we denote the symmetric group of $k$ letters.

\section{Preliminaries}

	\subsection{Basis of $\mathfrak{sl} (n+1)$} 
	Let $e_{i,j}$  stand for the $(i,j)$th elementary matrix of $\mathfrak{gl} (n+1)$. We fix the following basis of ${\mathfrak s} = \mathfrak{sl} (n+1)$:  
	$\left\{h_k, e_{i,j} \; | \; 1\leq i,j \leq n+1, i \neq j, k=1,...,n\right\}$, where $h_k := e_{k,k} - \frac{1}{n+1}\sum_{i=1}^{n+1} e_{i,i}$. In particular, $e_{i,i} - e_{j,j} = h_i-h_j$ if $1\leq i,j\leq n$ and  $e_{i,i} - e_{n+1,n+1} = h_i + \sum_{j=1}^n h_j$. Unless otherwise stated, whenever $e_{i,j}$ is used we assume that $i\neq j$. We also fix $\mathfrak b$ to be the Borel subalegbra of  ${\mathfrak s}$ spanned by $e_{i,j}$, $i<j$, and $h_k$, $k=1,..,n$, and $\mathfrak h = \bigoplus_{k=1}^n \mathbb C h_k$. We will say that a ${\mathfrak s}$-module $M$ is a \emph{weight module} if 
	$$
	M = \bigoplus_{\lambda \in \h^*} M^{\lambda},
	$$
	where $M^{\lambda} = \{m \in M \; | \; hm = \lambda(h)m, \mbox{ for every }h \in \h \}$. We call $\dim M^{\lambda}$ the \emph{$\lambda$-weight multiplicity} of $M$. If the set of all weight multiplicities is bounded we call the module \emph{bounded}. A weight module $M$ with finite weight multiplicities is called \emph{torsion free} if $e_{i,j}$ act injectively (hence, bijectively) on all weight spaces. Torsion free modules are particular examples of bounded modules since all weight multiplicities are equal. 
	
	 We note that in some cases will consider weight modules of $\mathfrak{sl} (n+1)$ relative to other Cartan subalgebras $\h'$.

	\subsection{Some automorphisms of $\mathfrak{sl} (n+1)$} 
	
	By $\tau$ we will denote the negative transpose on  $\mathfrak{gl} (n+1)$, i.e. $\tau(e_{i,j}) = -e_{j,i}$ and we use the same letter for the restriction of $\tau$ on $\mathfrak{sl} (n+1)$. Then  $\tau$ is an involutive automorphism. For ${\bf a} \in \left( {\mathbb C}^*\right)^{n+1}$, we set $\varphi_{\bf a} (e_{i,j}) = \frac{a_i}{a_j}e_{i,j}$. Then $\varphi_{\bf a}$ and $\tau$ are automorphisms of $\mathfrak{sl} (n+1)$ and  $\varphi_{\bf a} = \varphi_{\bf a'}$ is and only if ${\bf a'}  = c {\bf a}$ for some $c \in {\mathbb C}^*$.  By $F_{\tau}$ and $F_{\bf a}$ we denote the endofunctors on $\mathfrak g$-mod corresponding to the twists by $\tau$ and $\varphi_{\bf a}$, respectively. 
	
\subsection{Fourier transform on ${\mathcal D} (n)$}
The Fourier transform on  ${\mathcal D} (n)$ is an automorphism of ${\mathcal D} (n)$ defined by a subset $S$ of $\{1,2,...,n \}$ as follows:
\begin{eqnarray*}
\psi_S(t_i)  & = &  \partial_i, \; \psi_S(\partial_i)  =  -t_i,   \mbox{ if } i \in S,\\
\psi_S(t_j)  & = &  t_j, \; \psi_S(\partial_j)  =   \partial_j, \mbox{ if } j \notin S.
\end{eqnarray*}
In other words, $\psi_S$ is the $S$-induced map from the Fourier transform $t \mapsto \partial$, $\partial \mapsto -t$ on ${\mathcal D} (1)$.  If $M$ is a ${\mathcal D} (n)$-module, by $M^{\psi_S}$ we will denote the module obtained from $M$ after twisting by $\psi_S$.  We note that $\psi_S^4 = \mbox{Id}$.

\subsection{Exponentiation on ${\mathcal D} (n)$}
For an arbitrary polynomial $g \in \mathcal O$, we define the automorphism $\theta_{g}$ of ${\mathcal D} (n)$ via $\theta_{g} (t_i) = t_i$,  $\theta_{g} (\partial_i) = \partial_i + \frac{ \partial g}{ \partial t_i}$, for $i=1,...,n$.
We will call $ \theta_{g}$, the \emph{$g$-exponentiation} on ${\mathcal D} (n)$. Since $\theta_{g} = \theta_{g+c}$, \emph{we will assume that $g\in {\mathcal O}_0$ whenever $\theta_{g}$ is considered}.

 If $M$ is a ${\mathcal D} (n)$-module, by $M^{{\rm exp}_{g}}$ we will denote the modules obtained from $M$ after twisting by $\theta_{g}$. Alternatively,  $M^{\exp_g}$ can be thought as the space $M e^{g}$ with the natural action of $\mathcal D (n)$. In the special case when $g$ is a homogeneous linear polynomial $g = \sum_{i=1}^n b_it_i$, we will denote $ \theta_{g}$ and $M^{{\rm exp}_{g}}$ by  $\theta_{b}$ and  $M^{\exp_b}$, respectively, where ${b} = (b_1,...,b_n)$ is in ${\mathbb C}^n$.

\subsection{Twisted localization of ${\mathcal D} (n)$-modules and  $\mathfrak{sl} (n+1)$-modules}
We first recall some properties of the twisted localization functor in general. Let $\mathcal U$ be an associative unital algebra and $\mathcal H$ be a commutative subalgebra of $\mathcal U$. We assume in addition that $\mathcal H =  \C [\mathfrak h]$ for some vector space ${\mathfrak h}$, and that 
$$\mathcal U=\bigoplus_{\mu\in {{\mathfrak h}^*}}\mathcal U^\mu,$$
where
$$\mathcal U^\mu=\{x\in\mathcal U | [h,x]=\mu(h)x, \forall h\in\mathfrak h\}.$$

Let  $a$ be an ad-nilpotent element of $\mathcal U$. Then the set $\langle a \rangle = \{ a^n \; | \; n \geq 0\}$ is an Ore subset of $\mathcal U$  which allows us to define the  $\langle a \rangle$-localization $D_{\langle a \rangle} \mathcal U$ of $\mathcal U$. For a $\mathcal U$-module $M$  by $D_{\langle a \rangle} M = D_{\langle a \rangle} {\mathcal U} \otimes_{\mathcal U} M$ we denote the $\langle a \rangle$-localization of $M$. Note that if $a$ is injective on $M$, then $M$ is isomorphic to a submodule of $D_{\langle a \rangle} M$. In the latter case we will identify $M$ with that submodule.

We next recall the definition of the generalized conjugation of $D_{\langle a \rangle} \mathcal U$ relative to $x \in {\mathbb C}$. This is the automorphism  $\phi_x : D_{\langle a \rangle} \mathcal U \to D_{\langle a \rangle} \mathcal U$ given by $$\phi_x(u) = \sum_{i\geq 0} \binom{x}{i} \ad (a)^i (u) a^{-i}.$$ If $x \in \mathbb Z$, then $\phi_x(u) = a^xua^{-x}$. With the aid of $\phi_x$ we define the twisted module $\Phi_x(M) = M^{\phi_x}$ of any  $D_{\langle a \rangle} \mathcal U$-module $M$. Finally, we set $D_{\langle a \rangle}^x M = \Phi_x D_{\langle a \rangle} M$ for any $\mathcal U$-module $M$ and call it the \emph{twisted localization} of $M$ relative to $a$ and $x$. We will use the notation $a^x\cdot m$  (or simply $a^x m$) for the element in  $D_{\langle a \rangle}^x M$ corresponding to $m \in D_{\langle a \rangle} M$. In particular, the following formula holds in $D_{\langle a \rangle}^{x} M$:
$$
u (a^{x} m) = a^{x} \left( \sum_{i\geq 0} \binom{-x}{i} \ad (a)^i (u) a^{-i}m\right)
$$
for $u \in \mathcal U$, $m \in D_{\langle a \rangle}  M$.

We will apply the twisted localization functor for  $({\mathcal U}, {\mathcal H})$ in the following three cases:

 \noindent (i) $\mathcal U = {\mathcal D}(n)$,   $\mathfrak h = \bigoplus_{i=1}^n \left( \mathbb C \medskip x_i\partial_i \right)$;

 \noindent (ii) $\mathcal U = U(\mathfrak{sl}(n+1))$, $\mathfrak h =  \bigoplus_{i=1}^n \left( \mathbb C h_i \right)$; 

 \noindent (iii) $\mathcal U = U(\mathfrak{gl}(n))$, $\mathfrak h =  \bigoplus_{i=1}^n \left( \mathbb C E_{ii}\right)$.

In  case (i), for simplicity,  we will use the following notation:
$D_i^+ = D_{\langle t_i \rangle}$, $D_i^- = D_{\langle \partial_i \rangle }$. Also, for  $\mathcal U = {\mathcal D}(n)$ and a $\mathcal U$-module $M$, we set $D_{(i)}^+M = (D_i^+\mathcal U  / \mathcal U ) \otimes_{\mathcal U} M$ and $D_{(i)}^-M = (D_i^-\mathcal U  / \mathcal U ) \otimes_{\mathcal U} M$. In the particular case, when $t_i$ (respectively, $\partial_i$) acts injectively on $M$, then  $D_{(i)}^+M \simeq  D_i^+M/M$ (respectively, $D_{(i)}^-M \simeq  D_i^-M/M$). Also, we set $D_S^{+} = \prod_{i \in S} D_i^{+}$ and $D_{(S)}^{+} = \prod_{i \in S} D_{(i)}^{+}$.

In case (ii), we will often consider the following setting. If $\Sigma$ is a  set of commuting roots (i.e. $\alpha, \beta \in \Sigma$ implies $\alpha + \beta \notin \Sigma$) and $f_{\alpha} \in {\mathfrak s}^{-\alpha}$ for $\alpha \in \Sigma$, then we consider $D_{\Sigma} = \prod_{\alpha \in \Sigma} D_{\langle f_{\alpha} \rangle}$. Also, if $\Sigma$ is a linearly independent set, and $\mu=\sum_{\alpha} \mu_{\alpha} \alpha$, then we set  $D_{\Sigma}^{\mu} = \prod_{\alpha \in \Sigma} D_{\langle f_{\alpha} \rangle}^{\mu_{\alpha}}$.
 
\subsection{Modules of Nilsson} \label{subsec-nilsson}
Recall that  $\mathfrak h$ is the Cartan subalgebra of $\mathfrak g$ spanned by $h_i$, $i=1,...,n$. We 
identify $U(\mathfrak h)$ with ${\mathbb C} [{\bf h}] = {\mathbb C} [h_1,...,h_n]$.  Let $\sigma_i \in \Aut( {\mathbb C} [h] ) $ be defined by $\sigma_i( f(h_1,...,h_n))= f(h_1,...,h_i -1 ,...,h_n) $,  $i \in \llbracket n \rrbracket$. Then, following  \cite{Nil1}, for $S \subset \llbracket n \rrbracket$ and $b \in \mathbb{C}$ we define the $\mathfrak s$-module $M_{b}^{S}$ as follows. The underlying space of  $M_{b}^{S}$  is  $\mathbb{C}[{\bf h}]$ and the ${\mathfrak s}$-action is defined by 
	
	\begin{displaymath}
	\begin{array}{rcl}
	h_{k}\cdot f&:=&h_{k}f, \qquad  k\in \llbracket n \rrbracket;\\
	\\
	e_{i,n+1}\cdot f&:=&\begin{cases}
	(h_1+...+h_n + b)\sigma_{i}f, & i \in S,\\
	(h_1+...+h_n + b)(h_{i}-b-1)\sigma_{i}f, & i\not\in S;
	\end{cases}
	\\
	\\
	e_{n+1,j}\cdot f&:=&\begin{cases}
	-(h_{j}-b) \sigma_{j}^{-1}f, & j \in S,\\
	-\sigma_{j}^{-1}f,  & j\not\in S;
	\end{cases}
	\\
	\\
	e_{i,j}\cdot f&:=&\begin{cases}
	(h_{j}-b)    			 \sigma_{i}\sigma_{j}^{-1}f, & i,j \in S,\\
	\sigma_{i}\sigma_{j}^{-1}f, & i \in S, j\not\in S,\\
	(h_{i}-b-1)(h_{j}-b)    		 \sigma_{i}\sigma_{j}^{-1}f, & i \not\in S,j \in S,\\
	(h_{i}-b-1)    			 \sigma_{i}\sigma_{j}^{-1}f, & i,j\not\in S.\\
	\end{cases}
	\end{array}
	\end{displaymath}
One of the main results in \cite{Nil1} is that the modules $F_{\bf a} (M_{b}^{S})$ and  $F_{\tau}F_{\bf a} (M_{b}^{S})$, for  $S \subset \llbracket n \rrbracket$, ${\bf a}  \in \C^n$, form a skeleton in the category  of $\mathfrak g$-modules which are free of rank $1$ when restricted to $U(\mathfrak h)$. In this paper we will present the list of these modules as a particular case of exponential tensor modules. 
	
	\subsection{Weighting functor and coherent families} Coherent families of weight modules were introduced by O. Mathieu in \cite{M}. These families play crucial role in Mathieu's classification of all simple torsion free modules of ${\mathfrak s}$. 
	\begin{definition}
	 Let $U({\mathfrak s})^0$ be the centralizer of $\h$ in $U({\mathfrak s})$. A \emph{coherent ${\mathfrak s}$-family of degree $d$} is a weight ${\mathfrak s}$-module $\mathcal M$ such that:
	 \begin{itemize}
	 \item[(i)] $\dim \mathcal M^{\lambda} = d$ for every $\lambda \in \h^*$
	 \item[(ii)]  For any $u\in  U({\mathfrak s})^0$, the map $\lambda \mapsto \Tr \left(u|  \mathcal M^{\lambda} \right)$ is polynomial in $\lambda$.
	 \end{itemize}

	\end{definition}
	Note that since finitely generated bounded modules have finite length (Lemma 3.3 in \cite{M}), we can define the semisimplification $\mathcal M^{\rm ss}$ of a coherent family $\mathcal M$. Namely,  $\mathcal M^{\rm ss} = \bigoplus_{\lambda \in \mathfrak h^*/ \Z \Delta}\mathcal M [\lambda]^{\rm ss}$, where  $\mathcal M [\lambda]= \bigoplus_{\alpha \in \Z \Delta} \mathcal M^{\lambda + \alpha}$. 
	
	The definition of weighting functor $\mathcal{W}$ appeared first in \cite{Nil1} attributing the idea to O. Mathieu. For any module $M$ over ${\mathfrak s} = \mathfrak{sl}(n+1)$, the \emph{weighting $\mathcal{W}(M)$} of $M$ is a coherent family defined as follows. Let ${\rm Max}\, U(\h) $ denote the set of maximal ideals of $U(\h)$. Also, for $\lambda \in \h^*$ by $\overline{\lambda}: U(\h)  \to \C$ we denote the algebra homomorphism such that  $\overline{\lambda}|{\h} = \lambda$. Then 
	$$
	\mathcal{W}(M) := \bigoplus_{\mathfrak{m} \in {\rm Max}\, U(\h)} M/\mathfrak{m}M = \bigoplus_{\lambda \in \mathfrak{h}^*} M/ \ker(\overline{\lambda})M
	$$
	has an ${\mathfrak s}$-module structure via the action	 $x_{\alpha} \cdot (v+ \ker(\overline{\lambda})M) := (x_{\alpha} \cdot v) + \ker(\overline{\lambda+\alpha})M$, where $x_{\alpha}$ is in the $\alpha$-root space of ${\mathfrak s}$.

\subsection{Families of differential operator presentations of $\mathfrak{sl} (n+1)$ }
Recall that we assume $i \neq j$ whenever we write $e_{i,j}$.
\begin{proposition} \label{omega-v-s}
Let   $V$ be a $\mathfrak{gl} (n)$-module and  $S$  be a subset of $\llbracket n \rrbracket$. Then the correspondence
\begin{eqnarray*}
h_k &\mapsto & -t_k\partial_k \otimes 1 + 1\otimes E_{kk} -1\otimes 1, \mbox{ for }k\notin S\\
h_k  &\mapsto &  t_k\partial_k \otimes 1 + 1 \otimes E_{kk}, \mbox{ for }k\in S\\
e_{i,j} &\mapsto & 1 \otimes E_{ij} -t_j \partial_i \otimes 1,  \mbox{ for }  i,j \notin S\\
e_{i,j} &\mapsto & 1 \otimes E_{ij} + t_i\partial_j \otimes 1,  \mbox{ for }  i,j \in S\\
e_{i,j} &\mapsto & 1 \otimes E_{ij} + t_it_j \otimes 1,  \mbox{ for } i\in S, j\notin S\\
e_{i,j} &\mapsto & 1 \otimes E_{ij}  - \partial_i\partial_j \otimes 1,  \mbox{ for } i\notin S, j\in S\\
e_{n+1,j} &\mapsto &-t_j \otimes 1,  \mbox{ for } j\notin S\\
e_{n+1,j} &\mapsto & -\partial_j \otimes 1,  \mbox{ for } j\in S\\
e_{i,n+1} &\mapsto &-\sum_{j\notin S}\partial_j \otimes E_{ij} + \sum_{l \in S} t_l \otimes E_{il} + \sum_{j\notin S} t_j\partial_j\partial_i \otimes 1 - \sum_{l \in S}t_l \partial_l\partial_i \otimes 1 - \sum_{j=1}^{n} \partial_i \otimes E_{jj} \\
&& + ((n+1) - |S|)\partial_i \otimes 1,  \mbox{ for } i \notin S\\
e_{i,n+1} &\mapsto &- \sum_{j\notin S}\partial_j \otimes E_{ij} + \sum_{l \in S} t_l \otimes E_{il} - \sum_{j\notin S} t_it_j\partial_j \otimes 1 + \sum_{l \in S} t_it_l\partial_l \otimes 1 + \sum_{j=1}^{n} t_i \otimes E_{jj} \\
&& - (n -|S|)t_i \otimes 1,  \mbox{ for } i\in S.
\end{eqnarray*}
extends to a homomorphism $\omega_{V,S} : \mathfrak{sl} (n+1) \to {\mathcal D}(n) \otimes \End(V)$.
\end{proposition}
\begin{proof}
The case when $S =   \llbracket n \rrbracket$ has been known for long time and usually is attributed to Rudakov, \cite{R}, and Shen, \cite{Sh}. The case of arbitrary $S$ follows from  $S =   \llbracket n \rrbracket$ by applying appropriate Fourier transform. Namely,  $\omega_{V,S} = (\psi_{\widehat{S}}^3 \otimes 1) \omega_{V, \llbracket n \rrbracket }$, where $\widehat{S} = \llbracket n \rrbracket \setminus S$.
\end{proof}
\subsubsection{The case $S=\emptyset$} \label{subsec-empty} We will denote $\omega_{V,\emptyset}$ by $\omega_{V}$. In this case we have that 
\begin{eqnarray*}
h_k &\mapsto & -t_k\partial_k \otimes 1 + 1\otimes E_{kk} -1\otimes 1, \mbox{ for all }k,\\
e_{i,j} &\mapsto & 1 \otimes E_{ij} -t_j \partial_i \otimes 1,  \mbox{ for all }  i \neq j, \\
e_{n+1,j} &\mapsto &-t_j \otimes 1,  \mbox{ for all } j,\\
e_{i,n+1} &\mapsto &-\sum_{j=1}^n \partial_j \otimes E_{ij} + \sum_{j=1}^n  t_j\partial_j\partial_i \otimes 1 - \sum_{j=1}^{n} \partial_i \otimes E_{jj} + (n+1)\partial_i \otimes 1,  \mbox{ for all } i.
\end{eqnarray*}
One easily checks that $\omega_{V,S} = (\psi_S \otimes 1) \omega_{V}$. Furthermore, the above correspondence define a homomorphism $\omega : U(\mathfrak{sl} (n+1)) \to {\mathcal D}(n) \otimes U(\mathfrak{gl} (n))$ that will play important role in Section \ref{sec-tran-fun}.
\subsubsection{The case $S=\llbracket n \rrbracket$} In this other ``extreme'' case we have the following presentation:
\begin{eqnarray*}
h_k &\mapsto & t_k\partial_k \otimes 1 + 1\otimes E_{kk}, \mbox{ for all }k,\\
e_{ij} &\mapsto & 1 \otimes E_{ij} + t_i\partial_j \otimes 1,  \mbox{ for all }  i \neq j, \\
e_{n+1,j} &\mapsto & -\partial_j \otimes 1,  \mbox{ for all } j,\\
e_{i,n+1} &\mapsto &  \sum_{l=1}^n t_l \otimes E_{il} + \sum_{l =1}^n t_it_l\partial_l \otimes 1 + \sum_{j=1}^{n} t_i \otimes E_{jj}, \mbox{ for all } i.
\end{eqnarray*}
\subsection{Exponential tensor modules}
For a $ \mathcal D(n)$-module $P$, a subset $S$ of  $\llbracket n \rrbracket$, and a  $\mathfrak{gl}(n)$-module $V$, by $T(P,V)$ we denote the space $P \otimes V$ considered as a module over ${\mathfrak s} = \mathfrak{sl}(n+1)$ through the homomorphism $\omega_{V} = \omega_{V,\emptyset}$. In particular, the ${\mathfrak s}$-module 
with underlying space $P \otimes V$ obtained from the homomorphism $\omega_{V,S}$ is isomorphic to $T(P^{\psi_S},V)$. 

We will pay special attention at the case when $P = (\mathcal O^{\psi_S})^{\exp_g} = e^{g} \left( \mathbb C[\bf t] \right)^{\psi_S}$, where $g$ is a polynomial in $\mathbb C[\bf t]$. In this case, we will call 
$$
T(g, V,S) = T\left( \left(  \left( \mathbb C[{\bf t}]\right)^{\psi_S}\right)^{\exp_g}, V, \emptyset\right) = T(e^{g} \left( \mathbb C[\bf t] \right)^{\psi_S}, V) .
$$
\emph{exponential tensor module} corresponding to $g$, $V$, and $S$. In the case when $g = \sum_{i=1}^nb_nt_n$, we set $T(b, V,S) = T(g, V,S) $ where $b \in {\mathbb C}^n$.  In the case when $g=0$ the modules $T(0, V,S)$ can be considered as Fourier transforms of the classical tensor modules studied originally by Rudakov, Shen, and others.  The modules $T(0, V,S)$ play important role in the classification of simple torsion free ${\mathfrak s}$-modules of Mathieu, \cite{M}, as they are parts of coherent families defined in the next subsection. If $g \neq 0$, the modules $T(g, V,S)$ are  not weight modules, as the following statement shows. The proof follows directly from the definition of $\omega_V$, see \S \ref{subsec-empty}.

\begin{lemma} \label{lem-weight}
The module $T(g, V,S)$ is a weight $\mathfrak{sl}(n+1)$-module if and only if $g=0$ and $V$ is a weight $\mathfrak{gl}(n)$-module.
\end{lemma}

\subsection{Tensor coherent families} \label{ten-coh-fam}

Let $V$ be a finite dimensional $\mathfrak{gl}(n)$-module and let $S \subset \llbracket n \rrbracket$. 
Note that the space ${\mathcal F}_{\lambda} = {\bf t}^{\lambda} \C [{\bf t}^{\pm 1}]$ has a natural structure of a $\mathcal D (n)$-module. Moreover, ${\mathcal F}_{\lambda} = {\mathcal F}_{\mu}$ if and only if $\lambda - \mu \in \Z^n$. Hence, we may d efine  ${\bf t}^{\lambda} \C [{\bf t}^{\pm 1}]$ for $\lambda \in {\mathbb C}^n/\Z^n$. Let $\mathcal T (V,S) = \bigoplus_{\lambda \in {\mathbb C}^n/\Z^n} T(\left( {\bf t}^{\lambda} \C [{\bf t}^{\pm 1}] \right)^{\psi_S}, V)$. Then one easily checks the following.
\begin{proposition}
The ${\mathfrak s}$-module $\mathcal T (V,S)$ is a coherent family of degree $\dim V$.
\end{proposition}

\section{Localization of exponential tensor modules}
In this section we obtain some important results on localization of the exponential tensor modules $T(g,V,S)$. These results will help us to establish simplicity criteria for $T(g,V,S)$ in some particular cases of $V$.

\begin{lemma} \label{lem-loc-exp}
If $P = {\mathbb C}[{\bf t}]$ is the defining representations of ${\mathcal D}(n)$, then
$$
e^g P^{\psi_S} \simeq D_{(S)}^+ (e^gP) \simeq e^g D_{(S)}^+ (P).
$$
\end{lemma}
\begin{proof}
The second isomorphism is straightforward. For the first one we may assume for simplicity that $g=0$. Let $\widehat{S} =  \llbracket n \rrbracket \setminus S$. Using the multi-index notation, the space $D_{(S)}^+ P$ has basis $t_S^{\bf m}t_{\widehat{S}}^{\bf \ell}$, where ${\bf m} \in \left( \mathbb Z_{<0}\right)^{|S|}$ and ${\bf \ell} \in \left( \mathbb Z_{\geq 0}\right)^{|\widehat{S}|}$. For ${\bf k} = (k_1,...,k_{|S|})$ and $p \in {\mathbb C}[{\bf t}]$ with $\partial_ip = 0$ for all $i \in S$, consider the map $t_S^{\bf k}p \mapsto \partial_S^{\bf k}(t_S^{\bf -1})p$. It is not difficult to check that this map extends to a homomorphism $P^{\psi_S}  \to D_{(S)}^+ (P) $. This is an isomorphism since it maps a basis element to a nonzero scalar multiple of the corresponding basis element. 
\end{proof}
\subsection{The case of one-dimensional $V$} \label{subsec-1-dim}
We now focus on the case when $V$ is one-dimensional representation of weight $a \in {\mathbb C}$. We denote this representation by $V_a$. In other words  $V_a = a \tr$.

\begin{proposition} \label{prop-v-a} Given $g \in \mathcal O_0$, $a \in \mathbb C$, and $S  \subset \llbracket n \rrbracket $, we have the following.
\begin{itemize}
\item[(i)] If $(n+1)(a-1) \notin \mathbb Z$, then $T(g,V_a,S)$ is simple. 
\item[(ii)] If $(n+1)(a-1) \in \mathbb Z_{ \leq -n-1}$, then $T(g,V_a,S)$ is simple if and only if  $S = \emptyset$.
\item[(iii)]  If $(n+1)(a-1) \in \{-n,-n+1,...,-1\}$, then $T(g,V_a,S)$ is simple if and only if  $S = \emptyset$ or $S =   \llbracket n \rrbracket$.
\item[(iv)] If $(n+1)(a-1) \in \mathbb Z_{ \geq 0}$, then $T(g,V_a,S)$ is simple if and only if  $S = \llbracket n \rrbracket$.
\end{itemize} 
\end{proposition}
\begin{proof} By Lemma \ref{lem-loc-exp}, $T(g,V_a,S) \simeq T(e^g D_{(S)}^+ (\mathbb C [{\bf t}]), V_a)$. The module $T(e^g D_{(S)}^+ (\mathbb C [{\bf t}]), V_a)$ has a basis $e^g t^{\bf m}$, where $m_i \in \mathbb Z_{<0} $ for $i \in S$ and  $m_i \in \mathbb Z_{\geq 0} $ for $i \notin S$. 

We next prove the ``only if'' part for all (i)-(iv).  First assume  that $(n+1)(a-1) \in \mathbb Z$. The coefficient $b_{\bf m}$ of $e^g {\bf t}^{\bf m}t_i^{-1}$ in the expansion of  $ e_{i,n+1} (e^g {\bf t}^{\bf m}) $ is 
 \begin{equation}\label{eqn-coeff}
 b_{\bf m} = m_i \left(\sum_{j=1}^n m_j -1 - (n+1) (a-1)\right).
 \end{equation}
From here we easily check that 
$$
T'= \Span \{ e^g {\bf t}^{\bf m} \; | \; \sum_{j=1}^n m_j \geq  (n+1)(a-1) +1\}
$$
is a submodule of $T(g,V_a,S)$. We have that $T'=T(g,V_a,S)$ if and only if $(n+1)(a-1) \in {\mathbb Z}_{\leq -1}$ and $S = \emptyset$; and $T'=0$  if and only if $(n+1)(a-1) \in {\mathbb Z}_{\geq  -n}$ and $S = \llbracket n \rrbracket$. This completes the proof of the ``only if'' statements. 

It remains to show the ``if'' parts. Assume that the conditions for $S$ in (i)-(iv) are satisfied and let  $M$ be a nontrivial submodule of $T(g,V_a,S)$.  We first show \emph{homogeneity of $M$}, namely, if $ e^g  f \in M$,  $ f = \sum_{\bf m}a_{\bf m}{\bf t}^{\bf m}$,   then $e^g{\bf t}^{\bf m} \in M$, whenever $a_{\bf m} \neq 0$. Indeed, for $k=1,....,n$ we have that
$$
h_k e^g  f =  -t_k\partial_kg e^g  f -  e^g t_k\partial_kf + (1-a) e^g  f.
$$
Thus $e^g t_k\partial_kf \in M$ which easily implies $e^g{\bf t}^{\bf m} \in M$ if  $a_{\bf m} \neq 0$.  Using the homogeneity of $M$ and applying multiple actions of $e_{i,n+1}$ and  $e_{n+1,j}$  if necessary,  we obtain that  $e^g{\bf t}_S^{\bf -1} \in M$. Then using again  homogeneity of $M$ and  multiple actions of $e_{i,n+1}$ and  $e_{n+1,j}$, we see that  $e^g{\bf t}^{\bf m} \in M$ for all ${\bf m}$ such that $m_i \in \mathbb Z_{<0} $ for $i \in S$ and  $m_i \in \mathbb Z_{\geq 0} $ for $i \notin S$. Hence, $M = T(g,V_a,S)$.\end{proof}
\begin{remark}
The theorem above provides a correction to Theorem 32(ii) in \cite{Nil1}, see Corollary \ref{cor-nilsson} for the transition between the modules of Nilsson and $T(g,V_a,S)$.
\end{remark}

\subsection{The case of $V = \bigwedge^k{\mathbb C}^n$}
Here we focus on the case when $V$ is an exterior power of the natural module of $\mathfrak{gl} (n)$. We will use the notation $V = \bigwedge^k{\mathbb C}^n$ in this case. The next result gives simplicity criterion for $T(g, \bigwedge^k{\mathbb C}^n ,S)$. The simplicity of  $T(g, \bigwedge^k{\mathbb C}^n ,S)$ as modules over the Lie algebra $W_n$ was proven in \cite{LLZ}.

\begin{proposition}  \label{prop-exterior}Given $g \in \mathcal O_0$, $k \in \{0,1,...,n \}$, and $S  \subset \llbracket n \rrbracket $,  we have the following.
\begin{itemize}
\item[(i)] If  $0<k<n$,  then $T(g, \bigwedge^k{\mathbb C}^n ,S)$ is not simple.
\item[(ii)]  If $k=0$, then $T(g, \bigwedge^k{\mathbb C}^n ,S)$ is  simple if and only if $S =\emptyset$.
\item[(iii)] If $k=  n $, then $T(g, \bigwedge^k{\mathbb C}^n ,S)$ is  simple if and only if $S = \llbracket n \rrbracket$.
\end{itemize} 
\end{proposition}
\begin{proof} The cases $k=0$ and $k=n$ correspond to the cases $a=0$ and $a=1$ in \S \ref{subsec-1-dim}. This proves parts (ii) and (iii). Part (i) follows from Lemma 3.2 in \cite{LLZ}, but for reader's convenience we outline the important parts in the proof. The crucial part is that for any ${\mathcal D}(n)$-module $P$, there is a differential map:
\begin{equation} \label{diff-map}
d_P : T(P, \bigwedge {\mathbb C}^n) \to T(P, \bigwedge {\mathbb C}^n),
\end{equation}
$d_P(f\otimes v) = \sum_{i=1}^n (t_i f) \otimes (e_i \wedge v)$, where $(e_1,...,e_n)$ is the standard basis of ${\mathbb C}^n$. This map has the property that $d_P^2=0$ and that it is $\mathfrak{sl}(n+1)$-homomorphism (in fact, it is $W_n$-homomorphism, too). This leads to the de Rham complex 
\begin{equation} \label{de-rham}
0 \xrightarrow[]{d_P}T(P, \bigwedge\nolimits^0 {\mathbb C}^n) \xrightarrow[]{d_P}  T(P, \bigwedge\nolimits^1 {\mathbb C}^n) \xrightarrow[]{d_P} \cdots \xrightarrow[]{d_P} T(P, \bigwedge\nolimits^n {\mathbb C}^n) \xrightarrow[]{d_P}  0.
\end{equation}

 Thus $d_P\left[ T(P, \bigwedge^{k-1} {\mathbb C}^n) \right]$ is a nontrivial proper submodule of $T(P, \bigwedge^{k}{\mathbb C}^n)$ for $k=1,...,n-1$. Letting $P = e^g\left(\mathbb C [{\bf t}]\right)^{\psi_S}$ leads to the proof of the nonsimplicity of $T(g, \bigwedge^k{\mathbb C}^n ,S)$  for $k=1,...,n-1$.  \end{proof}

\begin{remark}
Consider the case $S = \llbracket n \rrbracket$. Then the map $d_P$ defined in \eqref{diff-map} is nothing but the standard de Rham differential. In particular, if $g=0$, we have that $T(g, \bigwedge^k {\mathbb C}^n, \llbracket n \rrbracket)$ is isomorphic to the module $\Omega_{\mathbb C^n}^k$ of $k$-forms on ${\mathbb C^n}$ and $d_P=d$ is the standard  differential operator on the de Rham complex on  ${\mathbb C^n}$. In the case of arbitrary $g$, $d_P = d_g$ is the Witten deformation of the standard de Rham differential by $g$ defined in \cite{W}. Namely,  $d_P(\omega) = d(\omega)+dg\wedge\omega$.
\end{remark}

\begin{remark}
The structure of the (non-simple) modules $T(g, \bigwedge^k{\mathbb C}^n ,S)$, $0<k<n$, is a bit complicated.  The modules may have length 2, 4, or 5. A Jordan-H\"older decomposition of  $T(g, \bigwedge^k{\mathbb C}^n ,S)$, $0<k<n$, will be provided in \cite{Ng}.
\end{remark}

\section{Central characters and translation functors} \label{sec-tran-fun}
Recall the homomorphism $\omega : U(\mathfrak{sl} (n+1)) \to {\mathcal D}(n) \otimes U(\mathfrak{gl} (n))$ defined in \S \ref{subsec-empty}. In this section we identify $\mathfrak{gl} (n)$ with the subalgebra $\mathfrak{g}_n$ of $\mathfrak{sl} (n+1)$ spanned by $e_{ij}$, $1\leq i \neq j \leq n$, $h_k$, $k=1,...,n$. In particular, $\mathfrak h$ is the fixed Cartan subalgebra of $\mathfrak{sl} (n+1)$ and  $\mathfrak{gl} (n)$. Recall that we identify $U(\mathfrak h)$ with ${\mathbb C}[{\bf h}]$. We set $\mathfrak b_{\g} = \mathfrak b \cap \mathfrak g$.

\subsection{Central characters of the exponential tensor modules}
As usual, for a reductive Lie algebra $\mathfrak a$, we say that an $\mathfrak a$-module has central character $\chi : Z(\mathfrak a) \to \mathbb C $ if every $z \in  Z(\mathfrak a) $ acts on $M$ as $\chi(z) {\rm Id}$. For $\lambda \in \h^*$, we set $\chi_{\lambda}$ to be the central character of the simple $\mathfrak b$-highest weight $\mathfrak{sl} (n+1)$-module $L_{\mathfrak b} (\lambda)$ with highest weight $\lambda$. The character of the corresponding simple $\mathfrak b_{\g}$-highest weight $\mathfrak{g}_n$-module will be denoted by $\chi'_{\lambda}$.

We identify $\h^*$ with ${\mathbb C}^n$ and set $\rho_{\mathfrak s} = \left( n,n-1,...,1 \right)$,  $\rho_{\g} = \left(n-1, n-2,...., 0\right)$, and ${\bf 1} = \left(1,1,...,1 \right)$. In particular, $\rho_{\mathfrak s} = \rho_{\g} + {\bf 1}$. The elements in $\lambda \in \h^*$ naturally extend to algebra homomorphisms $\lambda : \mathbb C [{\bf h}] \to \C$. Then there exist isomorphisms $\xi_{\mathfrak s} : Z(\mathfrak{sl} (n+1)) \to \C [{\bf h}]$ and $\xi_{\g} : Z(\mathfrak{g}_n) \to \C [{\bf h}]$ such that $\chi_{\lambda} (z) = (\lambda + \rho_{\mathfrak s}) (\xi_{\mathfrak s} (z))$ and $\chi'_{\lambda} (z') = (\lambda + \rho_{\g}) (\xi_{\g} (z'))$, respectively,  for all $\lambda \in \h^*$, $z \in Z(\mathfrak{sl} (n+1))$, and $z' \in Z(\mathfrak{g}_n)$. The maps  $\xi_{\mathfrak s}$ and $\xi_{\g}$ are certainly the restrictions of the Harish-Chandra homomorphisms $U(\mathfrak{sl}(n+1))^{\h} \to U(\h)$ and $U(\mathfrak{g}_n)^{\h} \to U(\h)$ to the corresponding centers, where $U(\mathfrak a)^{\h}$ stands for the centralizer of $\h$ in $U(\mathfrak a)$. Let $\xi : Z(\mathfrak{sl} (n+1)) \to Z(\mathfrak{g}_n)$ be the composition $\xi = \xi_{\g}^{-1} \xi_{\mathfrak s}$. 
\begin{proposition}
With the notation as above, $\omega (z) = 1 \otimes ( \xi (z) ) $ for all $z \in Z(\mathfrak{sl} (n+1)) $.
\end{proposition}

\begin{proof} We will prove that the identity $\omega (z) = 1 \otimes \xi (z)$ holds when both sides are considered as endomorphisms on  $T(\lambda) = T(0,\C [{\bf t}], L_{\mathfrak b_{\g}}(\lambda))$ for $\lambda \in \C^n$. This is sufficient since $T(\lambda) = \C [{\bf t}] \otimes  L_{\mathfrak b_{\g}}(\lambda)$ is a faithful module over ${\mathcal D} (n) \otimes U(\mathfrak{gl} (n))$. Let $v_0$ be a $\mathfrak b_{g}$-highest weight vector of $L_{\mathfrak b_{\g}}(\lambda)$. Using the formulas in \S \ref{subsec-empty}, one easily checks that  $1 \otimes v_0$ is a $\mathfrak b$-highest weight vector of the 
$\mathfrak{sl} (n+1)$-module $T(\lambda)$. Again by these formulas,  the weight of $1 \otimes v_0$ is $\lambda - {\bf 1}$. Therefore, $\omega(z) = \chi_{\lambda - {\bf 1}} (z) {\rm Id}$ on $T(\lambda)$. On the other hand, if $z \in Z(\mathfrak{sl} (n+1))$, then 
\begin{equation} \label{equ-chars}
1 \otimes \xi (z) =  \chi_{\lambda}(\xi(z)) {\rm Id} = (\lambda + \rho_{\g}) (\xi_{\mathfrak s} (z)) = (\lambda - {\bf 1}+ \rho_{\mathfrak s}) (\xi_{\mathfrak s} (z)) = \chi_{\lambda - {\bf 1}} (z) {\rm Id}
\end{equation}
on $T(\lambda)$. This completes the proof.
\end{proof}

\begin{corollary} \label{cor-char}
If $V$ is a $\mathfrak{gl} (n)$-module of central character $\chi_{\lambda}$, and $P$ is a ${\mathcal D}(n)$-module, then $T(P,V)$ has central character $\chi_{\lambda - {\bf 1}}$. 
\end{corollary}
\begin{proof}
For brevity, write $T(V) = T(P,V)$. Then 
$$
\chi_{T (V)} (\omega (z)) {\rm Id} = \omega(z) = 1\otimes \xi (z) = \chi_V(\xi (z))  {\rm Id}  = \chi_{\lambda } (z')   {\rm Id} = \chi_{\lambda - {\bf 1}} (z)   {\rm Id}.
$$
The last identity follows from the identities \eqref{equ-chars} in the proof of the last proposition.  \end{proof}

\begin{corollary}
If $V_1$ and $V_2$ are simple finite-dimensional $\mathfrak{gl} (n)$-modules, $g_1,g_2 \in {\mathcal O}_0$ and $S_1,S_2 \subset \llbracket n \rrbracket $, then $T(g_1,V_1,S_1) \simeq T(g_2,V_2,S_2)$ if and only if $g_1=g_2$, $V_1\simeq V_2$, and $S_1 = S_2$.
\end{corollary}
\begin{proof}
The isomorphism $V_1\simeq V_2$ follows from Corollary \ref{cor-char}. Then after a $\theta_{-g_2}$-twist, we obtain $T(g_1 - g_2,V_1,S_1) \simeq T(0,V_1,S_2)$. From Lemma \ref{lem-weight} we have $g_1=g_2$. To conclude that $S_1=S_2$, we note for example that $e_{n+1,j}$ acts locally nilpotently on $T(0,V,S)$ if and only if $j \in S$.
\end{proof}
We note that an isomorphism theorem for tensor modules over the Lie algebra $W_n$ is established in \cite{LLZ} (Lemma 3.7).

\subsection{Translation  functors and exponential tensor modules}
Recall that we have fixed $\h$ to be the Cartan subalegbra of $\mathfrak s = \mathfrak{sl} (n+1)$ and $\g \simeq \mathfrak{gl} (n)$ and $\mathfrak b_{\mathfrak s}$ and $\mathfrak b_{\g}$ to be the corresponding Borel subalgebras. We will often write the weights of $\mathfrak{s}$ and $\g$ as $n$-tuples.  

For simplicity we set  $L_{\g}(\lambda) =  L_{\mathfrak b_{\g}} (\lambda)$ and  $L_{\mathfrak s} (\lambda) =  L_{\mathfrak b_{\mathfrak s}} (\lambda)$. Let $\Lambda_{\g}^+$ (respectively,  $\Lambda_{\g}^+$) denote the sets of weights in $\h^*$ such that $L_{\g}(\lambda)$ (respectively, $L_{\mathfrak s} (\lambda)$) is finite dimensional. In other words, $(\lambda_1,...,\lambda_n)\in \Lambda_{\g}^+$ (respectively, $(\lambda_1,...,\lambda_n)\in \Lambda_{\mathfrak s}^+$ ) if and only if $\lambda_{i} - \lambda_{i+1} \in \Z_{\geq 0}$ for $i=1,...,n-1$ (respectively,  $(\lambda_1,...,\lambda_n)\in \Lambda_{\g}^+$ and $\lambda_n+\sum_{i=1}^n\lambda_i \in \Z_{\geq 0}$).

In view of the Weyl group action, for $\lambda \in \h^*$, it is convenient to define $\lambda_{n+1} = - \sum_{i=1}^n\lambda_i$   and  $\widetilde{\lambda} = (\lambda_1,...,\lambda_n,\lambda_{n+1})$. Let $s_i$ denote the transposition $(i,i+1)$ and let $w_k = s_n s_{n-1}...s_{n-k+1}$. In other words, we choose $w_k$ to be the minimal $S_n$-coset representative of length $k$. By definition, $w_0=\mbox{Id}$. Let also, $w_k'$ be the transposition $(n-k+1,n)$, i.e.   $w_k' = s_n s_{n-1}...s_{n-k+1}...s_{n-1}s_n$. As usual, for $w \in S_{n+1}$ and $\widetilde{\lambda} \in \C^{n+1}$, we set $w\cdot \widetilde{\lambda} = w ( \widetilde{\lambda} + \rho_{\mathfrak s}) - \rho_{\mathfrak s}$. By definition, we have $w(\lambda) = w(\widetilde{\lambda})$ and $w \cdot \lambda = w \cdot \widetilde{\lambda}$. We let  $\omega_k = \sum_{i=1}^k \varepsilon_i$ be the $k$-th fundamental weight. In particular, $ \omega_n = {\bf 1}$.

We represent  $\Lambda_{\g}^+$ as a disjoint union of three sets, $\Lambda_{\g}^+  =  \mathcal N \sqcup  \mathcal S \sqcup \mathcal R $, where:
\begin{itemize}
\item[(i)] $\lambda \in  \mathcal N$, if $\lambda \in \Lambda_{\g}^+$,  $\lambda_n+\sum_{i=1}^n\lambda_i \notin \Z$,
\item[(ii)] $\lambda \in  \mathcal S$, if $\lambda \in \Lambda_{\g}^+$, $\lambda_n+\sum_{i=1}^n\lambda_i \in \Z$, and $w_k' \cdot {\lambda} ={\lambda}$ for some $k$,
\item[(iii)] $\lambda \in  \mathcal R$, if $\lambda \in \Lambda_{\g}^+$, $\lambda_n+\sum_{i=1}^n\lambda_i \in \Z$, and $w_k' \cdot {\lambda} \neq {\lambda}$ for all $k$.
\end{itemize}
The sets $ \mathcal N$, $ \mathcal S$, and $ \mathcal R$,  are nothing else, but the sets of $\lambda$ such that $\widetilde{\lambda} $ is nonintegral, singular, and regular integral, respectively.  We further decompose $ \mathcal S$ and $ \mathcal R$ as follows:
$$
 \mathcal S = \sqcup_{i=1}^n {\mathcal H}^{i-1,i}, \;  \mathcal R = \sqcup_{i=0}^n {\mathcal H}^{i}, 
$$
where $\lambda \in {\mathcal H}^{i-1,i}$ if $w_i' \cdot {\lambda} = {\lambda}$;  $\lambda \in {\mathcal H}^{0}$ if $\lambda_n+\sum_{i=1}^n\lambda_i \in \Z_{\geq 0}$; and  $\lambda \in {\mathcal H}^{i}$, $i>1$, if ${\lambda} = w_i \cdot {\mu}$ for some $\mu \in  {\mathcal H}^{0}$. In particular, $ \mathcal H^0 = \Lambda_{\mathfrak s}^+$. Also, if $\lambda = 0$, then $w_k\cdot 0 =  \omega_{n-k}- {\bf 1}$.

The following is easy to verify and will be helpful for the simplicity criterion, i.e. Theorem \ref{thm-simple}. 

\begin{lemma} \label{lem-lambda-1} Let $\lambda \in \Lambda_{\g}^+$ and $\lambda_0 = \infty$. Then the following hold.
\begin{itemize}
\item[(i)] $\lambda - {\bf 1} \in  \mathcal N$, if and only if $\lambda_n+\sum_{i=1}^n\lambda_i \notin \Z$;
\item[(ii)] $\lambda - {\bf 1} \in \mathcal {H}^{k,k+1}$ if and only if  $\lambda_{n-k} - (n-k) = - \sum_{i=1}^n\lambda_i$, for $k = 0,...,n-1$;
\item[(iii)] $\lambda  - {\bf 1}  \in \mathcal {H}^{k}$  if and only if $
 \lambda_{n-k} - (n-k) >  - \sum_{i=1}^n\lambda_i  >   \lambda_{n-k+1} - (n-k+1)
$,
 for $k = 0,...,n$.
\end{itemize}
\end{lemma}

We now briefly recall  the definition of translation functors for $ {\mathfrak s}$. and  $ V $ be a
finite-dimensional $ {\mathfrak s}$-module, and let  $
\eta,\lambda\in{\mathfrak h}^{*} $ be such that $\lambda
-\eta \in \Supp V$. Let $ {\mathfrak s}^{\mu}$-mod
denote the category of $ {\mathfrak s} $-modules which admit a
generalized central character $ \chi_{\mu} $. The translation
functor $ T_V^{\eta, \lambda}:{\mathfrak s}^{\eta}\mbox{-mod} \to
{\mathfrak s}^{\lambda}\mbox{-mod}$ is defined by $
T_V^{\eta, \lambda} \left(M\right)=\left(M\otimes
V\right)^{\chi_\lambda} $, where $ \left(M\otimes
V\right)^{\chi_\lambda} $ stands for the direct summand of $
M\otimes V $ admitting generalized central character $
\chi_{\lambda} $. Assume in addition that $ \lambda - \eta$ belongs to the
$W$-orbit of the highest weight of $ V $,
the stabilizers of $ \eta+\rho $ and $
\lambda+\rho $ in the Weyl group coincide and $ \eta+\rho $, $
\lambda+\rho $ lie in the same Weyl facet. Then $
T_V^{\eta, \lambda}:{\mathfrak s}^{\eta}\mbox{-mod} \to {\mathfrak
s}^{\lambda}\mbox{-mod}$ defines an equivalence of categories (see
\cite{BG}). We will need the following particular case of the translation functor.

\begin{proposition} \label{prop-translation}
Let $\lambda \in \Lambda_{\g}^+$ and $P$ be a ${\mathcal D}(n)$-module. Then the following hold. 
\begin{itemize}
\item[(i)]  $\lambda  -{\bf 1} \in  \mathcal N$, then there is $a \in \C$ such that $(n+1) (a-1) \notin \Z$ and 
$T_V^{(a-1){\bf 1}, \lambda-{\bf 1}}T(P,V_a) \simeq T(P,L_{\g}(\lambda))$, where $V=L_{\mathfrak s}(\mu)$ with $\mu = \lambda - a {\bf 1}$.
\item[(ii)]  If  $\lambda -{\bf 1} \in \mathcal H^{k-1,k}$ and $a_k = \frac{k}{n+1}$, then $T_V^{(a_k-1){\bf 1}, \lambda-{\bf 1}}T(P,V_{a_k}) \simeq T(P,L_{\g}(\lambda))$, where $V=L_{\mathfrak s}(\mu)$ with $\widetilde{\mu} = (\lambda_1 - a_k, ...,\lambda_k - a_k, \lambda_k - a_k, ..., \lambda_n - a_k) $.

\item[(iii)] If $\lambda -{\bf 1} \in \mathcal H^{k}$, then  $T_V^{\omega_{n-k}-{\bf 1}, \lambda-{\bf 1}}T(P,\Lambda^{n-k} \C^n ) \simeq T(P,L_{\g}(\lambda))$, where $V=L_{\mathfrak s}(\mu)$ is such that $\lambda -{\bf 1} = w_k \cdot \mu$.
\end{itemize}
Moreover, the functors $T_V^{\eta -{\bf 1}, \lambda - {\bf 1}}$ used in (i)-(iii) are equivalences of categories.
\end{proposition}
\begin{proof}
We first note that $T(P,W) \otimes V \simeq T(P,W \otimes V)$, and also that an exact sequence of $\g$-modules 
$$
0 \to W_1 \to W \to W_2 \to 0
$$
leads to an exact sequence 
$$
0 \to T(P,W_1) \to  T(W,V) \to  T(P,W_2	) \to  0
$$
of $\mathfrak{sl}(n+1)$-modules.  These two statements are written as Remarks 2.1 and 2.2 in \cite{GS2} in more restrictive setting, but the proofs are essentially the same. The idea is to observe that the modules $T(P,V)$ are $({\mathfrak s}, \mathcal O)$-modules (i.e. ${\mathfrak s}$-modules and  $\mathcal O$-modules with compatible actions of ${\mathfrak s}$ and $\mathcal O$). More precisely, $T(P,V) = P\otimes_{\mathcal O} \widetilde{V}$, where $\widetilde{V} = \mathcal O \otimes V$. Note that  $\widetilde{V}$ can  be treated as the module of sections of the vector bundle $\mathcal V$ induced from $V$ on the affine open subset of the projective space $\mathbb P^n$ consisting of $[t_0,t_1,...,t_n]$ such that $t_i\neq 0$. For details on the geometric interpretation, see, for example,  \S 11 in \cite{M} or  \S 2 in \cite{GS2}.

We next apply the tensor product invariance and the exactness of $T(P,V)$ together with Corollary \ref{cor-char}. It remains to show that $L_{\g}(\eta) \otimes L_{\mathfrak s}(\mu)$ has a direct summand isomorphic to $L_{\g}(\lambda)$, and that there is no other direct summand $L_{\g}(\mu)$  such that $\chi_{\mu - {\bf 1}} = \chi_{\lambda - 1}$. Here $\eta = a{\bf 1}, a_k{\bf 1}, \omega_{n-k}$, for cases (i), (ii), (iii), respectively. This follows from the the Gelfand-Tsetlin decomposition of $L_{\mathfrak s}(\mu) = \bigoplus_{p} L_{\g}(\mu^{(p)})$ into $\g$-modules  $L_{\g}(\mu^{(p)})$ and  looking at the supports of the resulting products $L_{\g}(\eta) \otimes L_{\g}(\mu^{(p)})$. \end{proof}

\begin{remark}
A Jordan-H\"older decomposition of $T(P,V_a) \otimes V$, where $a \in \C$ and $V$ is a simple finite-dimensional $\mathfrak{sl}(n+1)$-module, is described explicitly in \cite{BL}. In particular, the authors prove their old conjecture that every simple torsion-free $\mathfrak{sl}(n+1)$-module is isomorphic to a  submodule of $T(P,V_a) \otimes V$, for particular choices of $P$, $a$, and $V$. One should note that for some choices of $a$, the corresponding translation functors are not equivalences of categories, and that is why it is better to work with  $T(P,\bigwedge^i \C^n)$.
\end{remark}

Combining Propositions \ref{prop-v-a}, \ref{prop-exterior}, and  \ref{prop-translation}, we obtain one of our main results in the paper. 
\begin{theorem} \label{thm-simple}
Let $\lambda \in \Lambda_{\g}^+$, $g \in \mathcal O_0$, and $S \subset \llbracket n \rrbracket$. Then the following hold.
\begin{itemize}
\item[(i)] If $\lambda - {\bf 1} \in  \mathcal N$, then $T(g, L_{\g}(\lambda), S)$ is simple.
\item[(ii)]  If $\lambda - {\bf 1} \in  \mathcal S$, then $T(g, L_{\g}(\lambda), S)$ is simple if and only if $S = \emptyset$ or $S = \llbracket n \rrbracket$.
\item[(iii)]  If $\lambda- {\bf 1}  \in  \mathcal R$, then: 
\begin{itemize}
\item[(a)] if $\lambda - {\bf 1} \in \mathcal H^0$, $T(g, L_{\g}(\lambda), S)$ is simple if and only if $S =  \llbracket n \rrbracket$. 
\item[(b)]  if $\lambda - {\bf 1} \in \mathcal H^i$, $1<i<n$, $T(g, L_{\g}(\lambda), S)$ is not simple,
\item[(c)]  if $\lambda - {\bf 1} \in \mathcal H^n$, $T(g, L_{\g}(\lambda), S)$ is simple if and only if $S = \emptyset$.
\end{itemize}
\end{itemize}
For the explicit conditions when $\lambda - {\bf 1} \in  \mathcal N$, etc. see Lemma \ref{lem-lambda-1}.
\end{theorem}

\section{Applications of exponential tensor modules in the case $\deg g \leq 1$}

\subsection{The case $g=0$: injective coherent families}
In this subsection we consider the case when $g=0$ and $V$ is a simple finite-dimensional $\mathfrak{gl}(n)$-module.

As mentioned earlier, coherent families are one of the main tools that O. Mathieu used in the classification of all simple torsion free modules of ${\mathfrak s}$ as the latter are submodules of partly-irreducible coherent families. We call a coherent family \emph{partly-irreducible} (or just \emph{irreducible}) if there is $\lambda \in \h^*$ such that $\mathcal M [\lambda] = \bigoplus_{\alpha \in \Z \Delta} \mathcal M^{\lambda + \alpha}$ is an irreducible ${\mathfrak s}$-module. The coherent families $\mathcal T (V,S)$ are partly-irreducible except for the case when $V = L_{\g}(\lambda)$ and $\lambda - {\bf 1} \in  \cup_{i=1}^{n}{\mathcal H^i}$. Below we provide more details about the non-irreducible case. For details we refer the reader to \S 11 in \cite{M}.

If $\mu \in {\mathcal H^0}$, we have a complex
$$
0 \to  \mathcal T (L_{\g}(w_n\cdot \mu + {\bf 1}), S) \to  \mathcal T (L_{\g}(w_{n-1}\cdot \mu + {\bf 1}), S)     \to \cdots \to \mathcal T (L_{\g}(w_{0}\cdot \mu + {\bf 1}), S) \to 0
$$
In the case $\mu =0$, we have $L_{\g}(w_{n-i}\cdot \mu + {\bf 1}) = L_{\g}(\omega_i) = \bigwedge^i (\mathbb C^n)$,   and the above complex is the de Rham complex defined in \eqref{de-rham}. The complex for arbitrary $\mu$ is obtained from the  complex for $\mu =0$ after applying the translation functor from Proposition \ref{prop-translation}(iii). 

Let now $\lambda - {\bf 1} \in {\mathcal H}^i$ for $0\leq i \leq n$, and  let  $\mu \in {\mathcal H}^0 $  be such that  $\lambda - {\bf 1} = w_i \cdot \mu$. Then we can define the coherent family 
$ \mathcal T' (L_{\g}(\lambda), S)  = \sum_{j=0}^i (-1)^{j} \mathcal T (L_{\g}(\lambda[j]), S)$,
where $\lambda[j] = w_{i-j}\cdot \mu + {\bf 1}$. Set for convenience $ \mathcal T' (L_{\g}(\lambda), S) =  \mathcal T (L_{\g}(\lambda), S)$ if $\lambda - {\bf 1} \notin \mathcal R$. 

Another important feature of the families $ \mathcal T (V, S) $ and  $ \mathcal T' (V, S)$ is their injectivity. For a subset $\Sigma$ of roots of $\mathfrak s$ and an ${\mathfrak s}$-module $M$, we say that $M$ is \emph{$\Sigma$-injective} if  every nonzero $\alpha$-root vector  $x_{\alpha}$ of $\mathfrak s$ acts injectively on $M$ for every $\alpha \in \Sigma$. Here is one important particular case of $\Sigma$. Let $\tilde{S}$ be a nonempty proper subset of $ \llbracket n+1 \rrbracket$, and let 
\begin{equation} \label{sigma-s}
\Sigma_{\tilde{S}} = \{ \varepsilon_i - \varepsilon_j \; | \; i \in \tilde{S}, j \notin \tilde{S} \}
\end{equation}

We say that an ${\mathfrak s}$-module $M$ is \emph{$\tilde{S}$-injective} if it is $\Sigma_{\tilde{S}}$-injective.

\begin{proposition} \label{prop-simple-fam}
Let $V$ be a simple finite-dimensional $\mathfrak{gl}(n)$-module and let $S$ be a proper subset of $\llbracket n \rrbracket$. Then $\mathcal T' (V,S)$ is an $\left(S \cup \{ n+1 \} \right)$-injective partly-irreducible coherent family, and  $F_{\tau} \mathcal T' (V,S)$ is $\left(\llbracket n\rrbracket \setminus S\right)$-injective partly-irreducible coherent family. 
\end{proposition}
\begin{proof}
The irreducibility of $\mathcal T' (V,S)$, and hence of  $F_{\tau} \mathcal T' (V,S)$, follows from Theorem 11.4 in \cite{M}. Although the latter theorem concerns the case  $S = \llbracket n\rrbracket$,  the case for arbitrary $S$ follows from the fact that the semisimplification $\mathcal T' (V,S)^{\rm ss}$ of $\mathcal T' (V,S)$ does not depend on the choice of  $S$. The injectivity follows by the  explicit formulas defining $\omega_{V,S}$ in Proposition \ref{omega-v-s}.
\end{proof}

Coherent families may be constructed also via twisted localization of highest weight modules. We  outline the construction of these families below and refer the reader to \S 4 in \cite{M} for details. 

If $\Sigma$ is a set of commuting roots that is a basis of $\Z \Delta$ and $M$ is a $\Sigma$-injective bounded simple module, then 
$$
\mathcal E_{\Sigma} (M) = \bigoplus_{\lambda \in \h^*/\Z\Delta} D_{\Sigma}^{\lambda} M.
$$
is a $\Sigma$-injective irreducible coherent family containing $M$ as a submodule. We call $\mathcal E_{\Sigma} (M) $ \emph{the $\Sigma$-injective coherent  extension of $M$}.  The injective coherent extensions are  unique in the sense of the following proposition.

\begin{proposition} \label{prop-inj-unique}Let  $\Sigma$ be a set of commuting roots that is a basis of $\Z \Delta$ and $M$ be a simple $\Sigma$-injective bounded ${\mathfrak s}$-module. Then $\mathcal E_{\Sigma} (M) $ is the unique, up to isomorphism,  $\Sigma$-injective irreducible coherent family containing a submodule isomorphic to $M$.
\end{proposition}
\begin{proof}
Let $\mathcal M$ be a  $\Sigma$-injective irreducible coherent family containing $M$ as a submodule and let $\lambda \in h^*$. Consider the modules $M_1 = \bigoplus_{\alpha \in \Z \Delta} \mathcal M^{\lambda + \alpha}$ and $M_2 =  D_{\Sigma}^{\lambda} M$. These modules have the property $(M_1)^{\rm ss} \simeq (M_2)^{\rm ss}$ by the uniqueness of the semisimple coherent extension (Proposition 4.8 in \cite{M}). Since $M_1$ and $M_2$ are both $\Sigma$-injective, there is a simple $\Sigma$-injective module $L$ that is isomorphic to submodules of $M_1$ and $M_2$. On the other hand, both $M_1$ and $M_2$ are dense, i.e. $\dim M_1^{\lambda + \alpha} = \dim M_2^{\lambda + \alpha} = \deg M$ for all $\alpha \in \Z \Delta$. Hence, $M_1 \simeq D_{\Sigma}L  \simeq M_1$, which implies the result.
\end{proof}

\begin{remark}
Britten and Lemire introduced slightly different notion of $\Sigma$-injective coherent family in \cite{BL}, see Definition 2.1. They imposed more restrictions on the action of the root elements on the family and as a result established a uniqueness result for the $\Sigma$-injective coherent families of degree $d$ containing a $\Sigma$-injective (not necessarily simple) bounded module $M$ of degree $d$. This uniqueness result was used to show that every simple torsion free module is a subquotient of $\mathcal M \otimes F$ where $\mathcal M$ is a coherent family of degree $1$ and $F$ is a finite-dimensional simple ${\mathfrak s}$-module.
\end{remark}

 Propositions \ref{prop-simple-fam} and \ref{prop-inj-unique} imply the following.

\begin{corollary} If $\tilde{S}$ is a nonempty proper subset of $\llbracket n+1 \rrbracket$, and $\mathcal M$ is a $\tilde{S}$-injective partly irreducible coherent family, then the following hold. 
\begin{itemize}
\item[(i)] If $n+1 \in \tilde{S}$, then  $\mathcal M \simeq \mathcal T' (V,S)$ for $S = \tilde{S} \setminus \{n+1 \}$ and some $V$.
\item[(ii)] If $n+1 \notin \tilde{S}$, then  $\mathcal M \simeq F_{\tau} \mathcal T' (V,S)$ for $S =  \llbracket n \rrbracket \setminus \tilde{S}$ and some $V$.
\end{itemize}
Furthermore, every simple infinite-dimensional bounded $\mathfrak s$-module is a submodule of $\mathcal T' (V,S)$ or  $F_{\tau} \mathcal T' (V,S)$ for some $V$ and $S$.
\end{corollary}

\begin{remark}
As mentioned in the proof of Proposition \ref{prop-simple-fam},  $\mathcal T' (V)^{\rm{ss}}  = \mathcal T' (V,S)^{\rm{ss}}$ is independent of the choice of $S$. In fact, if $L$ is an infinite-dimensional simple submodule of $\mathcal T' (V)^{\rm{ss}}$, then $\mathcal T' (V)^{\rm{ss}}$ is the unique, up to isomorphism, semisimple partly-irreducible coherent family containing $L$ as a submodule. In particular, if $L$ is $\Sigma$-injective, then $\mathcal T' (V)^{\rm{ss}} \simeq \mathcal E_{\Sigma}(L)^{\rm ss}$. Since the $\tilde{S}$-injective  partly-irreducible coherent families are  classified, it is natural to attempt to classify all  injective families $\mathcal M$ that are indecomposable, i.e.  such that all $\mathcal M [\lambda]$ are indecomposable modules. All $\mathcal T '(V,S)$ and $F_{\tau}\mathcal T' (V,S)$ will be in this classification list except two of them, $\mathcal T (V,\llbracket n \rrbracket)$ and $F_{\tau}\mathcal T (V,\llbracket n \rrbracket)$,  as they are not $\tilde{S}$-injective for any $\tilde{S}$.
\end{remark}

\subsection{The case $\deg g =1$ and $\h$-free modules of $\mathfrak{sl}(n+1)$}
In this subsection we fix $g (t) = \sum_{i=1}^nb_it_i$ for $b_i \in \C^*$. If ${\bf b} = (b_1,b_2,...,b_n)$, we will write $T({\bf b}, V,S)$ for $T(g, V,S)$ and ${\bf bt}$ for $g(t) = \sum_{i=1}^n b_it_i $.

\begin{theorem} \label{thm-h-free}
Let ${\bf b} \in \left( \C^{*} \right)^n$. Then the following correspondence defines a homomorphism $U({\mathfrak s}) \to \End (\C[{\bf h}] \otimes V)$:
\begin{eqnarray*}
h_k &\mapsto& h_k \otimes 1, \; \mbox{ for all } k, \\
e_{ij} &\mapsto& \frac{b_j}{b_i} h_i \sigma_i\sigma_j^{-1} \otimes 1 - \frac{b_j}{b_i} \sigma_i\sigma_j^{-1} \otimes E_{ii} + \sigma_i\sigma_j^{-1} \otimes E_{ij}, 	\; \mbox{ for } i,j \in S, \\
e_{ij} &\mapsto&  \frac{b_i}{b_j}(h_j +1) \sigma_i\sigma_j^{-1} \otimes 1 - \frac{b_i}{b_j}\sigma_i\sigma_j^{-1} \otimes E_{jj} + \sigma_i\sigma_j^{-1} \otimes E_{ij}, \; \mbox{ for }  i,j \notin S,\\
e_{ij}  &\mapsto&  -b_ib_j\sigma_i\sigma_j^{-1} \otimes 1 + \sigma_i\sigma_j^{-1} \otimes E_{ij}, \;\mbox{ for } i\notin S, j \in S, \\
e_{ij} &\mapsto& \frac{-1}{b_ib_j} h_i(h_j +1) \sigma_i\sigma_j^{-1} \otimes 1 + \frac{1}{b_ib_j}(h_j +1) \sigma_i\sigma_j^{-1} \otimes E_{ii} + \frac{1}{b_ib_j} h_i \sigma_i\sigma_j^{-1} \otimes E_{jj}  \\ && - \frac{1}{b_ib_j} \sigma_i\sigma_j^{-1} \otimes E_{ii}E_{jj} +\sigma_i\sigma_j^{-1} \otimes E_{ij}, \;\mbox{ for }  i \in S, j\notin S,\\
e_{n+1,j} &\mapsto&  -b_j\sigma_j^{-1} \otimes 1, \mbox{ for }  j \in S,\\
e_{n+1,j}  &\mapsto& \frac{1}{b_j} h_j\sigma_j^{-1} \otimes 1 - \frac{1}{b_j} \sigma_j^{-1} \otimes E_{jj} + \frac{1}{b_j} \sigma_j^{-1} \otimes 1, \mbox{ for }  j \notin S,\\
e_{i,n+1} &\mapsto&  \frac{1}{b_i}(\sum_{j=1}^n h_j -1)h_i \sigma_i \otimes 1 - \frac{1}{b_i}(\sum_{j=1}^n h_j)\sigma_i \otimes E_{ii} - \sum_{j\notin S} b_j \sigma_i \otimes E_{ij} + \sum_{p\in S} \frac{h_p}{b_p} \sigma_i \otimes E_{ip} \\ && -  \sum_{p\in S} \frac{1}{b_p} \sigma_i \otimes E_{ip}E_{pp} +  \sum_{p\in S} \frac{1}{b_p} \sigma_i \otimes E_{ip}, \; \mbox{ for }  i \in S,\\
e_{i,n+1} &\mapsto&  -b_i(\sum_{j=1}^n h_j -1) \sigma_i \otimes 1 - \sum_{j\notin S} b_j \sigma_i \otimes E_{ij} + \sum_{p\in S} \frac{h_p}{b_p} \sigma_i \otimes E_{ip} -  \sum_{p\in S} \frac{1}{b_p} \sigma_i \otimes E_{ip}E_{pp} \\ && +  \sum_{p\in S} \frac{1}{b_p} \sigma_i \otimes E_{ip},  \mbox{ for }  i \notin S.
	\end{eqnarray*}
	This homomorphism endows the space $\C[{\bf h}] \otimes V$ with an ${\mathfrak s}$-module structure. The resulting module is ${\h}$-free of rank $\dim V$, and it is isomorphic to $T({\bf b}, V,S)$.
	\end{theorem}
\begin{proof}
The homomorphism in the theorem is the composition of $\omega_{V,S}: U({\mathfrak s}) \to {\mathcal D} (n) \otimes \End (V)$ and the homomorphism  ${\mathcal D} (n) \otimes \End (V) \to \End (\C[{\bf h}] \otimes V)$ where the latter is defined by the following maps:

 \begin{eqnarray*}
1 \otimes E_{ij} &\mapsto& \sigma_i\sigma_j^{-1} \otimes E_{ij}, \mbox{ for all } i,j,\\
t_i \otimes 1&\mapsto& - \frac{1}{b_i}\left( (h_i+1) \sigma_i^{-1} \otimes 1 - \sigma_i^{-1} \otimes E_{ii} \right), \mbox{ for  } i \notin S, \\ 
\partial_i \otimes 1&\mapsto& b_i \sigma_i  \otimes 1, \mbox{ for  } i \notin S\\  
t_i \otimes 1&\mapsto& \frac{1}{b_i}\left( h_i \sigma_i \otimes 1 - \sigma_i \otimes E_{ii} \right), \mbox{ for  } i \in S,\\  
\partial_i \otimes 1&\mapsto& b_i \sigma_i^{-1} \otimes 1, \mbox{ for  } i \in S.
\end{eqnarray*}

To prove that the above define a homomorphism, we make a repeated use of the identity $h_i\sigma_i-\sigma_ih_i = \sigma_i$.

The isomorphism $T({\bf b}, V, S) \to \C[{\bf h}] \otimes V$ is given explicitly by the formulas
$$
e^{\bf bt}{\bf t}^{\bf k} \otimes v \mapsto \prod_{i\in S} \binom{h_i - {\rm wt}(v)_i}{k_i}\frac{k_i!}{b_i^{k_i}} \prod_{j\notin S} \binom{-h_j + {\rm wt}(v)_j-1}{k_j}\frac{k_j!}{b_j^{k_j}} \otimes v,
$$
where $v$ is a weight vector of $V$ of weight ${\rm wt}(v)$.
\end{proof}
The above theorem applied in the particular case of one-dimensional $V$ recovers the classification result of Nilsson discussed in \S \ref{subsec-nilsson}. Indeed we have the following. 
\begin{corollary} \label{cor-nilsson}
We have an isomorphism $F_{\bf b} M_a^S \simeq T\left({\bf b}_S, V_{a+1}, \llbracket n\rrbracket \setminus S\right)$, where $({\bf b}_S)_i = b_i$ for $i \in S$ and  $({\bf b}_S)_j = - \frac{1}{b_j}$ for $j \notin  S$. Hence, every $\h$-free ${\mathfrak s}$-module of rank $1$ is an exponential tensor module $T({\bf b}, V,S)$ or a $\tau$-twist of such module.
\end{corollary}

\begin{remark}
The maps  in the proof of Theorem \ref{thm-h-free} that define the homomorphism  ${\mathcal D} (n) \otimes \End (V) \to \End (\C[{\bf h}] \otimes V)$ can be considered as the composition $\beta\alpha$ of the maps $\alpha$ and $\beta$ defined as follows. Take for simplicity $S = \emptyset$ and $b_i = 1$ for all $i$. Then $t_i \mapsto \sigma_i$, $\partial_i \mapsto (h_i+1)\sigma_i^{-1}$ define a homomorphism $\beta: D_{\llbracket n \rrbracket}^+ {\mathcal D} (n) \to   \End (\C[{\bf h}])$. Note that $D_{\llbracket n \rrbracket}^+ {\mathcal D} (n)$ is nothing but the algebra of polynomial differential operators on $\C [t_1^{\pm 1},...,t_n^{\pm 1}]$. On the other hand $1 \otimes E_{ij} \mapsto t_it_j^{-1} \otimes E_{ij}$,  $t_i  \otimes 1 \mapsto - \partial_i \otimes 1 + t_i^{-1} \otimes E_{ii}$, $\partial_i \otimes 1 \mapsto t_i \otimes 1$, define an automorphism $\alpha$ of $D_{\llbracket n \rrbracket}^+{\mathcal D} (n) \otimes \End (V)$. It is worth noting that the homomorphism $\alpha \omega_{V,\emptyset}$ leads to another explicit presentation of $U({\mathfrak s})$ in terms of differential operators. This new presentation has the advantage that in the corresponding coherent family $\bigoplus_{\lambda \in \C^n}{\bf t}^{\lambda} \C [{\bf t}^{\pm 1}]\otimes V$, the weight vectors of weight $\lambda$ are precisely ${\bf t}^{\lambda} \otimes v$, $v \in V$.
\end{remark}
\begin{remark}
The isomorphism map  $T({\bf b}, V, S) \to \C[{\bf h}] \otimes V$  in the proof of Theorem \ref{thm-h-free} can be expressed in terms of generating series. Namely, by considering the coefficients of ${\bf x}^{\bf k}$ in the following map:
$$
\exp(\sum_{\ell=1}^nb_\ell t_\ell(x_\ell+1) ) \otimes v \mapsto \prod_{i \in S} (1+x_i)^{h_i-{\rm wt}(v)_i}  \prod_{j \notin S} (1+x_j)^{-h_j+{\rm wt}(v)_j-1} \otimes v
$$
\end{remark}

\begin{proposition}
The following isomorphism of coherent ${\mathfrak s}$-families hold: 
$${\mathcal W} (T({\bf b}, V, S) ) \simeq {\mathcal T}(V, \llbracket n\rrbracket \setminus S ).$$
\end{proposition}
\begin{proof}
Let $\dim V = N$ and let $(v_1,...,v_N)$ be a basis of weight vectors of $V$. Denote by ${\bf 1}_{\ell} \in  \C[{\bf h}]^{\oplus N}$  the element $(0,...,1,...,0)$, where $1$ is in the $\ell$th position. Let $v_{\lambda,\ell} = {\bf 1}_{\ell} + \ker(\bar{\lambda}) T({\bf b}, V, S)$ for $\lambda \in \h^*$ and $1 \leq \ell \leq N$. Then the explicit isomorphism ${\mathcal W} (T({\bf b}, V, S) ) \to {\mathcal T}(V, \llbracket n\rrbracket \setminus S )$ is given by the formulas
$$
v_{\lambda,\ell} \mapsto \prod_{i\in S} (b_it_i)^{-\lambda_i + {\rm wt}(v_\ell)_{i}  -1}\prod_{j\notin S} (b_jt_j)^{\lambda_j -{\rm wt}(v_\ell)_{j}} \otimes v_\ell.
$$

\end{proof}

\subsection{The case $\deg g = 1$ and $S =\emptyset$: weight modules relative to other Cartan subalelgebras} In this subsection we fix $\deg g = 1$ and $S =\emptyset$. As usual, we let $g(t) =  {\bf b} {\bf t} = \sum_{i=1}^nb_it_i$, but now we let $b_i$ be any (possibly zero) complex numbers. Consider the inner automorphism $\bar{\theta}_{\bf b} = \exp(- \ad(\sum_{j=1}^n b_j e_{n+1,j}))$ of ${\mathfrak s}$. In this case the module $T(g,V,\emptyset)$ is nothing but the $\bar{\theta}_{\bf b}$-twist of $T(V,\emptyset)$. In particular, we obtain the following.
\begin{corollary}
If $\deg g = 1$, then the modules $T({\bf b},V,\emptyset)$ are weight modules relative to the Cartan subalgebra $\h' = \bar{\theta}_{\bf b} (\h)$.
\end{corollary}

\subsection{The case $\deg g = 1$  and Whittaker modules} We conclude this section by identifying some exponential tensor modules as Whittaker modules. The classical Whittaker modules were introduced by Kostant in \cite{K} and correspond to the pair $({\mathfrak s}, \mathfrak n^+)$, where $\mathfrak n^+$ is the nilradical of a Borel subalegbra of ${\mathfrak s}$. More general Whittaker modules and pairs have been studied recently. In particular,  we will refer  to the constructions in \cite{BM} and \cite{BO} below. 

We fix once again $g(t) =  {\bf b} {\bf t}$. Let $S \subset  \llbracket n\rrbracket$ and $\mathfrak n_S$ be the Lie subalgebra of ${\mathfrak s}$ spanned by $e_{ij}$, $i \notin S$, $j \in S$. Note that $\mathfrak n_S$ is an abelian Lie algebra and it is the nilradical of a maximal parabolic subalgebra of ${\mathfrak s}$. Also, if $S$ is considered as a subset of $ \llbracket n+1\rrbracket$, the root system of $\mathfrak n_S$ is precisely $\Sigma_{S}$,  see \eqref{sigma-s}.

 Let us first consider $S = \emptyset$. Then the $\mathcal D (n)$-module $P = e^{{\bf b} {\bf t}} \mathbb C[\bf t]$ is a Whittaker module of type ${\bf b}$ in the sense of \cite{BO}. More precisely, the vector $w=e^{{\bf b} {\bf t}}$ satisfies the two properties: $\partial_i w = b_i w$ (i.e. $w$ is a Whittaker vector) and $P = \mathcal D (n) w$. However, as pointed out in \cite{Nil1}, the modules $T({\bf b},V,\emptyset)$ are not Whittaker modules. One reason for that is that they are free modules (of rank $\dim V$) over  $\mathfrak n_{\emptyset}$. In fact, every  $\mathfrak n_{\emptyset}$-free $\mathfrak s$-module of rank $1$ is isomorphic to $T({\bf b},V_a, \emptyset)$ for some ${\bf b}$ and $a$, see Theorem 3.2 in \cite{HCL}.

Let now $S \neq \emptyset$ and  set for convenience $b_{n+1} = 1$. Consider the homomorphism $\chi_{\bf b}: \mathfrak n_S \to \C$ defined by $\chi_{\bf b} (e_{ij}) = - b_ib_j$. Then $T(g,V,S)$ is a Whittaker module, in the sense of \cite{BM}. More precisely, $({\mathfrak s}, \mathfrak n_S)$ is a Whittaker pair and all vectors of the form $w = e^{{\bf b} {\bf t}} \otimes v$ with $E_{ij}v=0$ for $i \notin S$, $j \in S$, are Whittaker vectors of $T(g,V,S)$. Namely, $(x  - \chi_{\bf b}(x)) (w) = 0$ for $x \in \mathfrak n_S$.

\end{document}